\documentclass[11pt]{article}
\usepackage{amsmath}
\usepackage{amsfonts}
\usepackage{amsmath,amsthm}
\usepackage{amsmath,amssymb,amsthm,latexsym}
\usepackage{amscd}
\usepackage[all]{xy}
\usepackage{hyperref}
\usepackage{graphicx}

\newtheorem{theorem}{Theorem}[section]
\newtheorem{proposition}[theorem]{Proposition}
\newtheorem{definition}[theorem]{Definition}
\newtheorem{corollary}[theorem]{Corollary}
\newtheorem{lemma}[theorem]{Lemma}
\newtheorem{conjecture}[theorem]{Conjecture}
\newtheorem{example}[theorem]{Example}

\theoremstyle{remark}
\newtheorem{remark}[theorem]{Remark}

\begin{document}
\title{\bf{Complete stationary surfaces in $\mathbb{R}^4_1$ with total curvature $-\int K\mathrm{d}M=4\pi$}}

\author{Xiang Ma, Peng Wang}

\maketitle

\begin{center}
{\bf Abstract}
\end{center}

Applying the general theory about complete spacelike stationary
(i.e. zero mean curvature) surfaces in 4-dimensional Lorentz space $\mathbb{R}^4_1$,
we classify those regular algebraic ones with total Gaussian curvature
$-\int K\mathrm{d}M=4\pi$. Such surfaces must be oriented and be congruent to
either the generalized catenoids or the generalized enneper surfaces.
For non-orientable stationary surfaces, we consider the Weierstrass representation on the oriented double covering $\widetilde{M}$ (of genus $g$) and generalize Meeks and Oliveira's M\"obius bands.
The total Gaussian curvature are shown to be at least $2\pi(g+3)$ when $\widetilde{M}\to\mathbb{R}^4_1$ is algebraic-type.
We conjecture that there do not exist non-algebraic examples with $-\int K\mathrm{d}M=4\pi$.

\hspace{2mm}

{\bf Keywords:}  stationary surface, Weierstrass
representation, finite total Gaussian curvature,
singular end, non-orientable surfaces\\

{\bf MSC(2000):\hspace{2mm} 53A10, 53C42, 53C45}

\section{Introduction}
\label{intro}
In a previous paper \cite{Ma} we have generalized
the classical theory about minimal surfaces in $\mathbb{R}^3$ to
zero mean curvature spacelike surfaces in 4-dimensional Lorentz
space. Such an immersed surface $M^2\to\mathbb{R}^4_1$,
called a \emph{stationary surface} (see \cite{Alias} for related works before), admits a
Weierstrass-type representation formula, which involves a pair of meromorphic functions $\phi,\psi$ (the Gauss maps)
and a holomorphic $1$-form (the height differential) on $M$:
\[
{\bf x}= 2~\mathrm{Re}\int \Big(\phi+\psi, -\mathrm{i}(\phi-\psi),1-\phi\psi,1+\phi\psi\Big)\mathrm{d}h.
\]
Among complete examples, those with finite total curvature are most important, i.e., when the integral of the Gaussian curvature $-\int K\mathrm{d}M$ converges absolutely. For such surfaces,
under mild assumptions we have established Gauss-Bonnet type formulas
relating the total curvature with the Euler characteristic number of $M$,
the generalized multiplicities $\widetilde{d}_j$ of each ends,
the mapping degree of $\phi,\psi$, and the indices of the
so-called \emph{good singular ends}:
\begin{equation*}
%\begin{split}
\int_M K\mathrm{d}M=2\pi\Big(2-2g-r-\sum_{j=1}^r \widetilde{d}_j\Big)
=-2\pi \left(\deg\phi+ \deg\psi-\sum{_j} |\mathrm{ind}_{p_j}|\right)
%\end{split}
\end{equation*}
On this foundation, here we go on to consider complete examples with
total curvature $-\int K\mathrm{d}M=4\pi$, which is the smallest
possible value among algebraic stationary surfaces.
(Here we ingore the trivial case when $M$ is contained in a
3-dimensional degenerate subspace $\mathbb{R}^3_0$.
The induced metric is flat in that case with total curvature $0$. See Section~2.)

Recall that in $\mathbb{R}^3$, Osserman has shown that
complete minimal surfaces with finite total curvature
must be algebraic ones, i.e., they are given by meromorphic
Weierstrass data over compact Riemann surfaces.
In particular, immersed examples with $-\int K\mathrm{d}M=4\pi$
are either the catenoid or the Enneper surface.
(For other complete minimal surfaces in $\mathbb{R}^3$ with small total curvature $-\int K\mathrm{d}M\le 12\pi$ and the classification results, see \cite{Costa,Lopez}.)

These two classical examples have been generalized by us in \cite{Ma}
to stationary surfaces in $\mathbb{R}^4_1$
(see Example~\ref{exa-enneper} and \ref{exa-catenoid}). In this paper our main result is

\bigskip\noindent
{\bf Theorem A}~~ Let $x:M^2\to\mathbb{R}^4_1$ be
a complete, immersed, algebraic stationary surface
with total curvature $4\pi$. Then it is either
a generalized catenoid, or a generalized Enneper surface.
In particular, there does not exist non-orientable
examples with $-\int K\mathrm{d}M\le 4\pi$.

\bigskip
Compared to minimal surfaces in $\mathbb{R}^3$, here
finite total Gaussian curvature
(i.e., $\int_M K\mathrm{d}M$ converges absolutely)
still implies that $M$ is conformally equivalent to a compact
Riemann surface $\overline{M}$ with finite punctures $\{p_j|1\le j\le r\}$.
A main difference is that in our case, finite total curvature
no longer implies \emph{algebraic-type}.
For counter-examples see Example~\ref{exa-essen} and Example~\ref{exa-essen2}.
An interesting open problem is that whether there exist
non-algebraic examples with $-\int K\mathrm{d}M=4\pi$.
See discussions in Section~5.

Another new technical difficulty is that, to solve existence and
uniqueness problems for complete stationary surfaces, now we
must consider the following equation about complex variable $z$:
\begin{equation}\label{eq-singular}
\phi(z)=\bar\psi(z),
\end{equation}
We have to show that there are no solutions to it for
meromorphic functions $\phi,\psi$ with given
algebraic forms and certain parameters on a compact
Riemann surface $\overline{M}$ (except at several points
assigned to be \emph{good singular ends}).
This is because that on
an immersed surface there must be
$\phi\ne\bar\psi$ (\emph{regularity condition}). On the other hand, at one end where $\phi,\bar\psi$ take the same value
with equal multiplicities
(\emph{bad singular end}), the total curvature will diverge.
Such a complex equation \eqref{eq-singular}
involving both holomorphic and anti-holomorphic functions is quite unusual
to the knowledge of the authors. Most of the time we have to
deal with this problem by handwork combined with experience.
See \cite{Ma} or Appendix A for related discussions.
Note that $M\to\mathbb{R}^3$ is a rare case where we overcome
this difficulty easily, because this time $\phi\equiv -1/\psi$, and this will never be equal to $\bar\psi$.
\bigskip

In \cite{Meeks}, Meeks initiated the study of complete non-orientable minimal surfaces
in $\mathbb{R}^3$. Such surfaces are represented on its oriented
double covering space, and the example with least possible
total curvature $6\pi$ was constructed (Meeks' M\"obius strip).
Here we generalize this theory to non-orientable stationary
surfaces in $\mathbb{R}^4_1$ (Section~4).
A key result is the following lower bound estimation
of the total curvature which helps to establish Theorem~A above.

\bigskip\noindent
{\bf Theorem B}~~
Given a non-orientable surface $M$ whose double covering space \
$\widetilde{M}$ has genus $g$ and finite many ends,
for any complete algebraic stationary immersion
$x:M\to\mathbb{R}^4_1$ with finite total curvature there must be
$-\int_M K\mathrm{d}M\ge 2\pi(g+3).$
\bigskip

We conjecture that $2\pi(g+3)$ is the best lower bound which could always been attained.
Note that this agrees with the estimation for non-orientable minimal surfaces in $\mathbb{R}^3$,
and the conjecture is still open even in that special case \cite{Martin}.

We organize this paper as below. In Section~2 we review the
basic theory about stationary surfaces in $\mathbb{R}^4_1$.
The orientable case and non-orientable case
are discussed separately in Section~3 and 4.
In Section~5 we give non-algebraic examples with small total curvature.
The proofs to several technical lemmas are left to Appendix~A and B.

\bigskip\noindent
\textbf{Acknowledgement}~~
We thank two colleagues of the first author at Peking University,
Professor Fan Ding for providing the proof to the topological
Theorem~\ref{thm-odd} in Appendix~B, and Professor Bican Xia for
verifying Lemma~\ref{lem-main} in Appendix~A using a computational method developed by him before.
We also thank the encouragement of Professor Changping Wang.
This work is supported by the Project 10901006 of
National Natural Science Foundation of China.

\section{Preliminary}
Let ${\bf x}:M^2\to \mathbb{R}^4_1$ be an oriented complete
spacelike surface in 4-dimensional Lorentz space.
The Lorentz inner product $\langle\cdot,\cdot\rangle$ is given by
\[\langle {\bf x},{\bf x}\rangle=x_1^2+x_2^2+x_3^2-x_4^2.\]
We will briefly review the basic facts and global results
established in \cite{Ma} about such surfaces
with zero mean curvature (called \emph{stationary surfaces}).

Let $\mathrm{d}s^2=\mathrm{e}^{2\omega}|\mathrm{d}z|^2$
be the induced Riemannian metric on $M$ with respect to a
local complex coordinate $z=u+\mathrm{i}v$. Hence
\[
\langle {\bf x}_{z},{\bf x}_{z}\rangle=0,~~ \langle
{\bf x}_{z},{\bf x}_{\bar{z}}\rangle =\frac{1}{2}\mathrm{e}^{2\omega}.
\]
Choose null vectors ${\bf y},{\bf y}^*$ in the normal plane
at each point such that
\[
\langle {\bf y},{\bf y}\rangle=\langle {\bf y}^*,{\bf y}^*\rangle=0, ~~ \langle
{\bf y},{\bf y}^*\rangle =1,~~ \mathrm{det}\{{\bf x}_u,{\bf x}_v,{\bf y},{\bf y}^*\}>0~.
\]
Such frames $\{{\bf y},\ {\bf y}^{*}\}$ are determined up to scaling
\begin{equation}\label{scaling}
 \{{\bf y},\ {\bf y}^{*}\}\rightarrow
\{\lambda {\bf y},\ \lambda^{-1}{\bf y}^{*}\}
\end{equation}
for some non-zero real-valued function $\lambda$. After projection, we obtain two well-defined maps (independent to the scaling \eqref{scaling})
\[
[{\bf y}],\ [{\bf y}^* ]: M \rightarrow S^2\cong\{[{\bf v}]\in\mathbb{R}P^3|\langle {\bf v},{\bf v}\rangle=0\}.
\]
The target space is usually called the projective light-cone, which is well-known
 to be homeomorphic to the 2-sphere. By analogy to $\mathbb{R}^3$,
 we call them \emph{Gauss maps} of the spacelike surface ${\bf x}$ in $\mathbb{R}^{4}_{1}$.

The surface has zero mean curvature $\vec{H}=0$ if, and only if,
$[{\bf y}],\ [{\bf y}^* ]: M \rightarrow S^2$ are conformal mappings
(yet they induce opposite orientations on $S^2$).
Since $S^2\cong \mathbb{C}\cup\{\infty\}$,
we may represent them locally by a pair of
holomorphic and anti-holomorphic functions
$\{\phi,\bar{\psi}\}$. The Weierstrass-type
representation of stationary surface ${\bf x} : M\rightarrow \mathbb{R}^4_1$ is given by \cite{Ma}:
\begin{equation}\label{x}
{\bf x}= 2~\mathrm{Re}\int \Big(\phi+\psi, -\mathrm{i}(\phi-\psi),1-\phi\psi,1+\phi\psi\Big)\mathrm{d}h
\end{equation}
in terms of two meromorphic functions $\phi,\psi$ and
a holomorphic $1$-form $\mathrm{d}h$ locally.
 We call $\phi,\psi$ the Gauss maps of ${\bf x}$ and
$\mathrm{d}h$ the height differential.

\begin{remark}\label{rem-Wrepre}
When $\phi\equiv \mp 1/\psi$, by \eqref{x} we obtain
a minimal surface in $\mathbb{R}^3$, or a maximal surface in
$\mathbb{R}^3_1$. This recovers the Weierstrass representation
in these classical cases. When $\phi$ or $\psi$ is constant,
we get a zero mean curvature spacelike surface in the 3-space
$\mathbb{R}^3_0\triangleq \{(x_1,x_2,x_3,x_3)\in\mathbb{R}^4_1\}$
with an induced degenerate inner product, which is essentially
the graph of a harmonic function $x_3=f(x_1,x_2)$ on
complex plane $\mathbb{C}=\{x_1+\mathrm{i}x_2\}$.
\end{remark}

\noindent \textbf{Convention:}
In this paper, we always assume that neither of $\phi,\psi$
is a constant unless it is stated otherwise.
According to the remark above, we have ruled out
the trivial case of stationary surfaces in $\mathbb{R}^3_0$.
(According to \eqref{eq-totalcurvature} below,
such surfaces have flat metrics and zero total Gaussian curvature.)

\begin{remark}\label{rem-trans}
The induced action of a Lorentz orthogonal transformation of $\mathbb{R}^4_1$
on the projective light-cone is nothing but a M\"obius transformation on $S^2$, or equivalently,
a fractional linear transformation on $\mathbb{C}P^1=\mathbb{C}\cup\{\infty\}$
given by $A=\left(\begin{smallmatrix}a & b \\ c & d\end{smallmatrix}\right)$
 with $a,b,c,d \in \mathbb{C},~ad-bc=1$.
The Gauss maps $\phi,\psi$ and the height differential $\mathrm{d}h$ transform as below:
\begin{equation}\label{trans}
\phi\Rightarrow
\frac{a\phi+b}{c\phi+d}~,~~
 \psi\Rightarrow \frac{\bar{a}\psi+\bar{b}}{\bar{c}\psi+\bar{d}}~,~~
 \mathrm{d}h\Rightarrow (c\phi+d)(\bar{c}\psi+\bar{d})\mathrm{d}h~.
 \end{equation}
This is repeatedly used in Section~2 and Section~3 to simplify
or to normalize the representation of examples.
\end{remark}

\begin{theorem}\label{thm-period}\cite{Ma}
Given holomorphic $1$-form $\mathrm{d}h$ and meromorphic functions $\phi,\psi:M\rightarrow \mathbb{C}\cup\{\infty\}$
globally defined on a Riemann surface $M$. Suppose they satisfy the regularity condition
1),2) and period conditions 3) as below:

1) $\phi\neq\bar{\psi}$ on $M$ and their poles do not coincide;

2) The zeros of $\mathrm{d}h$ coincide with the poles of $\phi$ or $\psi$
with the same order;

3) Along any closed path the periods satisfy
\begin{equation}\label{eq-period1}
\oint_\gamma \phi \mathrm{d}h
=-\overline{\oint_\gamma \psi \mathrm{d}h }, ~~~(\text{horizontal period condition})
\end{equation}
\begin{equation}\label{eq-period2}
\mathrm{Re}\oint_\gamma \mathrm{d}h=\mathrm{Re}\oint_\gamma \phi\psi \mathrm{d}h=0.~~~(\text{vertical period condition})
\end{equation}
Then
\eqref{x} defines a stationary surface ${\bf x}:M\rightarrow \mathbb{R}^4_1$.

Conversely, any stationary surface ${\bf x}:M\rightarrow \mathbb{R}^4_1$ can be
represented as \eqref{x} in terms of such $\phi,\ \psi$ and $\mathrm{d}h$ over a (necessarily non-compact) Riemann surface $M$.
\end{theorem}

The structure equations and the integrability conditions are given
in \cite{Ma}. An extremely important corollary is the formula below
for the total Gaussian and normal curvature over a compact stationary surface $M$ with boundary $\partial M$:
\begin{equation}
\begin{split}
\int_M(-K+\mathrm{i}K^{\perp})\mathrm{d}M&=
2\mathrm{i}\int_M
\frac{\phi_z\bar{\psi}_{\bar{z}}}{(\phi-\bar{\psi})^2}
\mathrm{d}z\wedge \mathrm{d}\bar{z}
\\
&=-2\mathrm{i}\int_{\partial M} \frac{\phi_z}{\phi-\bar{\psi}} \mathrm{d}z
=-2\mathrm{i}\int_{\partial M}
\frac{\bar{\psi}_{\bar{z}}}{\phi-\bar{\psi}}\mathrm{d}\bar{z}.
\label{eq-totalcurvature}
\end{split}
\end{equation}

At one end $p$ with $\phi=\bar\psi$, the integral of
total curvature above will become an improper integral.
An important observation in \cite{Ma} is that this improper
 integral converges absolutely only for a special class of
such ends.

\begin{definition}
Suppose ${\bf x}:D-\{0\}\rightarrow \mathbb{R}^4_1$
is an annular end of a regular stationary surface (with boundary)
whose Gauss maps $\phi$ and $\psi$
extend to meromorphic functions on the unit disk $D\subset\mathbb{C}$.
It is called a \emph{\underline{regular end}} when
\[
\phi(0)\ne\bar\psi(0).~~~(\text{Thus}~ \phi(z)\ne\bar\psi(z),~\forall~z\in D.)
\]
It is a \emph{\underline{singular end}} if $\phi(0)=\bar\psi(0)$
where the value could be finite or $\infty$.

When the multiplicities of $\phi$ and $\bar\psi$ at $z=0$
are equal, we call $z=0$ a \emph{\underline{bad singular end}}.
Otherwise it is a \emph{\underline{good singular end}}.
\end{definition}

\begin{proposition}\cite{Ma}\label{prop-goodsingular}
A singular end of a stationary surface ${\bf x}:D-\{0\}\rightarrow \mathbb{R}^4_1$
is \emph{good} if and only if the curvature integral
\eqref{eq-totalcurvature} converges absolutely around this end.
\end{proposition}

For a good singular end we introduced the following
definition of its index.

\begin{definition}\cite{Ma}\label{lemma-index2}
Suppose $p$ is an isolated zero of $\phi-\bar\psi$ in $p$'s neighborhood $D_p$,
where holomorphic functions $\phi$ and $\psi$ take the value
$\phi(p)=\overline{\psi(p)}$ with multiplicity $m$ and $n$, respectively.
\emph{The index of $\phi-\bar{\psi}$ at $p$} (when $\phi,\psi$ are both holomorphic at $p$) is
\begin{equation}\label{eqind}
\mathrm{ind}_p(\phi-\bar{\psi})\triangleq
\frac{1}{2\pi\mathrm{i}}\oint_{\partial D_{p}}d\ln(\phi-\bar{\psi})=\left\{
   \begin{array}{ll}
          m, & \hbox{$m<n$;} \\
         -n, & \hbox{ $m>n$.}
   \end{array}
 \right.
\end{equation}
\emph{The absolute index of $\phi-\bar{\psi}$ at $p$} is
\begin{equation}\label{eqind+}
\mathrm{ind}^{+}_p(\phi-\bar{\psi})\triangleq
\left|\mathrm{ind}_p(\phi-\bar{\psi})\right|.
\end{equation}
For a regular end our index is still meaningful with
$\mathrm{ind}=\mathrm{ind}^+=0.$ For convenience we also introduce
\begin{equation}\label{eqind10}
\mathrm{ind}^{1,0}\!\triangleq \frac{1}{2}(\mathrm{ind}^{+}\!+\mathrm{ind}),~~~
\mathrm{ind}^{0,1}\!\triangleq \frac{1}{2}(\mathrm{ind}^{+}\!-\mathrm{ind}),
\end{equation}
which are always non-negative.
\end{definition}

Note that our definition of index of $\phi-\bar\psi$ is invariant
under the action of fractional linear transformation \eqref{trans}.
So it is well-defined for a good singular end of
a stationary surface.
In particular, we can always assume that our singular ends
do not coincide with poles of $\phi,\psi$;
hence the definition above is valid.

\medskip
A stationary surface in $\mathbb{R}^4_1$ is called an
\emph{algebraic stationary surface} if there exists a
compact Riemann surface $\overline{M}$
with $M=\overline{M}\backslash\{p_1,p_2,\cdots,p_r\}$ such that
${\bf x}_z \mathrm{d}z$ is a vector valued meromorphic form defined
on $\overline{M}$. In other words, the Gauss map $\phi,\psi$
and height differential $\mathrm{d}h$ extend to meromorphic functions/forms on $\overline{M}$.
For this surface class we have established
Gauss-Bonnet type formulas involving the indices of
the good singular ends.

\begin{theorem}[\cite{Ma}]\label{GB2}
For a complete algebraic stationary surface
${\bf x}:M\rightarrow \mathbb{R}^4_1$
given by \eqref{x} in terms of $\phi,\psi,\mathrm{d}h$
without bad singular ends, the total Gaussian curvature and total normal curvature are related with the indices at the ends $p_j$ (singular or regular) by the following formulas:
\begin{align}
\int_M K^{\perp}\mathrm{d}M &=0~,\label{eq-deg0}\\
\int_M K\mathrm{d}M &=-4\pi \left(\deg\phi-\sum{_j} \mathrm{ind}^{1,0}(\phi-\bar{\psi})\right) \label{eq-deg1}\\
&=-4\pi \left(\deg\psi-\sum{_j} \mathrm{ind}^{0,1}(\phi-\bar{\psi})\right),\label{eq-deg2}
\end{align}
From \eqref{eq-deg1}\eqref{eq-deg2} we have equivalent identities:
\begin{gather}
\sum{_j}\mathrm{ind}_{p_j}(\phi-\bar{\psi})=\deg\phi-\deg\psi~.\label{eq-deg3}\\
\int_M K\mathrm{d}M =-2\pi \left(\deg\phi+ \deg\psi-\sum{_j} \mathrm{ind}^{+}_{p_j}(\phi-\bar{\psi})\right)~.\label{eq-deg4}
\end{gather}
\end{theorem}

\begin{definition}\label{def-multiplicity}
The multiplicity of a regular or singular end $p_j$ for a stationary surface
in $\mathbb{R}^4_1$ is defined to be
\[
\widetilde{d}_j=d_j-\mathrm{ind}^+_{p_j},
\]
where $d_j+1$ is equal to the order of the pole of~
 ${\bf x}_z \mathrm{d}z$ at $p_j$.
\end{definition}

\begin{theorem}
[Generalized Jorge-Meeks formula \cite{Ma}]
\label{GB3}
Given an algebraic stationary surface ${\bf x}:M\rightarrow \mathbb{R}^4_1$
with only regular or good singular ends $\{p_1,\cdots,p_r\}=\overline{M}-M$.
Let $g$ be the genus of compact Riemann surface $\overline{M}$,
$r$ the number of ends, and $\widetilde{d}_j$ the multiplicity of $p_j$. We have
\begin{equation}\label{eq-jorgemeeks2}
\int_M
K\mathrm{d}M=2\pi\left(2-2g-r-\sum_{j=1}^r \widetilde{d}_j\right)~, ~~\int_M K^{\perp}\mathrm{d}M =0~.
\end{equation}
\end{theorem}

\begin{proposition}[\cite{Ma}]\label{ineq-multiplicity}
Let ${\bf x}:D^2-\{0\}\rightarrow \mathbb{R}^4_1$ be a regular or a good
singular end which is further assumed to be complete at $z=0$. Then
its multiplicity satisfies $\widetilde{d}\ge 1.$
\end{proposition}

\begin{corollary}
[The Chern-Osserman type inequality \cite{Ma}]
\label{GB2C}
Let ${\bf x}:M\rightarrow \mathbb{R}^4_1$ be an algebraic stationary surface without
bad singular ends, $\overline{M}=M\cup\{q_1,\cdots,q_r\}$. Then
\begin{equation}
\int K\mathrm{d}M \le 2\pi(\chi(M)-r)=4\pi(1-g-r).
\end{equation}
\end{corollary}
\begin{corollary}
[Quantization of total Gaussian curvature \cite{Ma}]
\label{cor-quantization}
Under the same assumptions of the theorem above,
when $\phi,\psi$ are not constants (equivalently, when ${\bf x}$ is not a
flat surface in $\mathbb{R}^3_0$), there is always
\[
-\int_M K\mathrm{d}M=4\pi k\ge 4\pi,
\]
where $k\ge 1$ is a positive integer.
\end{corollary}

\section{Orientable case and examples with $-\int K \mathrm{d}M=4\pi$}

This section is dedicated to the classification of complete
stationary surfaces immersed in $\mathbb{R}^4_1$ with finite
Gaussian curvature $-\int K \mathrm{d}M=4\pi$ which are
orientable and of algebraic type.

Under our hypothesis, the generalized Jorge-Meeks formula
\eqref{eq-jorgemeeks2} yields
\begin{equation}\label{4pi-end}
r+\sum\widetilde{d}_j+2g=4,
\end{equation}
and the index formulas \eqref{eq-deg1}\eqref{eq-deg2} read
\begin{equation}\label{4pi-ind}
\deg\phi-\sum\mathrm{ind}^{1,0}=1,~~\deg\psi-\sum\mathrm{ind}^{0,1}=1.
\end{equation}
Since $r\ge 1$, and $\widetilde{d}_j\ge 1$ for any end, there must be $g\le 1$,
and we need only to consider five cases separately as below.

\vspace{3mm}\noindent
$\bullet$~~\textbf{Case 1: $g=1, r=1, \widetilde{d}=1$ (torus with one end)}.

Since there is only one end, at least one of the indices $\mathrm{ind}^{1,0},\mathrm{ind}^{0,1}$ is zero.
By \eqref{4pi-ind} we know either $\phi$ or $\psi$ is a meromorphic function of degree $1$.
Yet this contradicts the well-known fact that over a torus there do not exist such functions.
So we rule out this possibility.

\vspace{3mm}\noindent
$\bullet$~~\textbf{Case 2: $g=0, r=1, \widetilde{d}=3$ and the unique end is regular}.

Such examples exist and they are generalization of the classical Enneper surface.

\begin{example}[The generalized Enneper surfaces]
\label{exa-enneper} This is given by
\begin{equation}\label{eq-enneper1}
\phi=z,\ \psi=\frac{c}{z}~,\ \mathrm{d}h=s\cdot z\mathrm{d}z,\
\end{equation}
or
\begin{equation}\label{eq-enneper2}
\phi=z+1,\ \psi=\frac{c}{z}~,\ \mathrm{d}h=s\cdot z\mathrm{d}z,
\end{equation}
over $\mathbb{C}$ with complex parameters $c,s\in\mathbb{C}\backslash\{0\}$.
${\bf x}$ has no singular points if and only if the parameter $c=c_1+\mathrm{i}c_2$
is not zero or positive real numbers in \eqref{eq-enneper1},
or
\begin{equation}\label{eq-enneper3}
c_1-c_2^2+\frac{1}{4}<0
\end{equation}
in \eqref{eq-enneper2}. When $c=-1$ in \eqref{eq-enneper1} we obtain
the Enneper surface in $\mathbb{R}^3$.
\end{example}

Indeed they are all examples in Case 2 according to the following result in \cite{Ma}.

\begin{theorem}\label{thm-enneper}\cite{Ma}
A complete immersed algebraic stationary surface in
$\mathbb{R}^4_1$ with $\int K=-4\pi$ and one regular
end is a generalized Enneper surface.
\end{theorem}

\vspace{3mm}\noindent
$\bullet$~~\textbf{Case 3: $g=0, r=1, \widetilde{d}=3$ with a good singular end}.

Suppose there exists such an example. Without loss of generality we assume that the singular end $p$
has positive index. Since $\mathrm{ind}\ge 1$, by definition we know that at $p$ the function $\psi$
takes the value $\psi(p)$ with multiplicity at least $2$. On the other hand, $\mathrm{ind}^{0,1}=0$
and $\deg\psi=1$, which is a contradiction to the observation above. Hence such examples do not exist.

\vspace{3mm}\noindent
$\bullet$~~\textbf{Case 4: $g=0, r=2, \widetilde{d}_j=1$
and both ends are regular}.

The classical catenoid is one of such examples. The generalization in $\mathbb{R}^4_1$ is
\begin{example}
[The generalized catenoids]
\label{exa-catenoid}
This is defined over $M=\mathbb{C}\backslash\{0\}$ with
\begin{equation}\label{eq-catenoid}
\phi=z+t,\ \psi=\frac{-1}{z-t},\ \mathrm{d}h=s\frac{z-t}{z^2}\mathrm{d}z.~~~~~
(-1<t<1, s\in\mathbb{R}\backslash\{0\})
\end{equation}
\end{example}
When $t=0$, it is the classical catenoid in $\mathbb{R}^3$.

\begin{theorem}\label{thm-catenoid}
\cite{Ma}
A complete immersed algebraic stationary surface in $\mathbb{R}^4_1$
with total curvature $\int K=-4\pi$ and two regular ends is
a generalized catenoid.
\end{theorem}

\vspace{3mm}\noindent
$\bullet$~~\textbf{Case 5: $g=0, r=2, \widetilde{d}_j=1$ with at least one good singular ends}.

This is the most difficult case in our discussion. We will show step by step that there are no such examples.

First, assume there is such a surface. We assert that it must have two singular ends whose indices
have opposite signs. Otherwise, if there is only one good singular end
which might be assumed to have positive index, similar to the discussion
in Case 3 we can show $\deg\psi=1$ and $\psi$ has multiplicity greater than $1$ at the end,
which is a contradiction. In the same way we can rule out the possibility that
both ends are singular with the same signs.

Second, without loss of generality we may suppose $M=\mathbb{C}\backslash\{0\}$ and the good singular ends are $0$ and $\infty$,
with $\mathrm{ind}_0=m\ge 1, \mathrm{ind}_\infty<0$.
By \eqref{4pi-ind}, $\mathrm{ind}^{1,0}=m, \deg\phi=m+1$.
If $\mathrm{ind}_\infty\le -m-1$, by definition we know
$\psi$ has multiplicity at least $m+1$ at $z=\infty$
where $\phi$ must has higher multiplicity, which is impossible since $\deg\phi=m+1$.
If $\mathrm{ind}_\infty\ge -m+1$, by definition and \eqref{4pi-ind} we know
$\mathrm{ind}^{0,1}\le m-1, \deg\psi\le m$, which contradicts the requirement that $\psi$ must
has multiplicity greater than $m$ at the first end $z=0$. In summary there must be
\begin{equation}\label{eq-case51}
\mathrm{ind}_0=m\ge 1,~~ \mathrm{ind}_\infty=-m,~~\deg\phi=\deg\psi=m+1\ge 2.
\end{equation}
We observe that $\phi(0)\ne\phi(\infty)$. Otherwise, since $z=0,\infty$ are both singular ends,
there must be $\psi(0)=\phi(0)=\phi(\infty)=\psi(\infty)$. Because $z=0$ is a good singular end
and $\mathrm{ind}_0=m$, $\psi$ has multiplicity at least $m+1$ at $z=0$ and multiplicity $m$ at $\infty$.
This is impossible when $\deg\psi=m+1, m\ge 1$.

This observation enables us to make the following normalization. Without loss of generality,
suppose $\phi(0)=\psi(0)=0,\phi(\infty)=\psi(\infty)=\infty$. Since meromorphic functions $\phi,\psi$
must be rational functions satisfying restrictions \eqref{eq-case51}, we know
\begin{equation}\label{eq-case52}
\phi(z)=z^m(z-a),~~\psi(z)=\frac{z^{m+1}}{z-b},~~\mathrm{d}h=\rho\frac{z-b}{z^k}\mathrm{d}z,
\end{equation}
where $a,b,\rho$ are arbitrary nonzero complex parameters. Note that $\mathrm{d}h$ takes
the form as above because $M$ is regular at $z=b$. On the other hand, at the ends $z=0$ and $z=\infty$
it should satisfy
$\widetilde{d}_0\ge 1,\widetilde{d}_{\infty}\ge 1$
according to Proposition~\ref{ineq-multiplicity},
which implies $k=m+2$ by the definition of $\widetilde{d}$.

After fixing the form of $\phi,\psi,\mathrm{d}h$, we verify the period conditions. It is easy to see
that the vertical period conditions are satisfied. The horizontal period conditions are satisfied if
and only if $a+b=-\bar\rho/\rho.$
In summary, such examples have Weierstrass data
\begin{equation}\label{eq-case53}
\phi(z)=z^m(z-a),~~\psi(z)=\frac{z^{m+1}}{z-b},
~~\mathrm{d}h=\rho\frac{z-b}{z^{m+2}}\mathrm{d}z,
\end{equation}
with parameters
\begin{equation}\label{eq-case54}
m\ge 1,~a,b,\rho\in \mathbb{C}\backslash{0},~a+b=-\bar\rho/\rho.
\end{equation}

If we can find nonzero parameters $a,b,\rho$ as above
so that the regularity condition $\phi\ne\bar\psi$
holds true for any $ z\in\mathbb{C}\cup\{\infty\}$,
then new examples with $-\int K\mathrm{d}M=4\pi$ are found.
But according to Lemma~\ref{lem-main} in Appendix A,
for any given nonzero parameters $a,b,\rho$ there always exist
nonzero solutions $z$ to the equation $\phi(z)=\overline{\psi(z)}$ for $\phi,\psi$ given above.
We conclude that there exist no examples in Case 5,
The proof to the following theorem has been finished.

\begin{theorem}\label{thm-4pi1}
Complete regular algebraic stationary surfaces
$x:M\to\mathbb{R}^4_1$ with $-\int K\mathrm{d}M=4\pi$
are either the generalized catenoids or the generalized
Enneper surfaces under the assumption that $M$ is orientable.
\end{theorem}

Another interesting observation is that if we make change of variables $z=w^2$ in \eqref{eq-case52},
and choose the power $k$ to be a even number suitably,
then the period conditions always hold true and we don't need the restriction $a+b=-1$ in
\eqref{eq-case54}. In this situation, if parameters $a=b$ is chosen suitably, the regularity
condition $\phi\ne\bar\psi$ is satisfied. See Lemma~\ref{lem-a=b}. In this way we find a
complete, immersed stationary surface in $\mathbb{R}^4_1$, yet with total curvature
$-\int K\mathrm{d}M=8\pi$. See the example below (which has appeared in \cite{Ma}).

\begin{example}
[Genus zero, two good singular ends and $\int_M K\mathrm{d}M=-8\pi$]
\label{exa-singular1}
\[
M=\mathbb{C}\backslash\{0\},~\phi=w^2(w^2-a),~\psi=\frac{w^4}{w^2-a},
~\mathrm{d}h=\frac{w^2-a}{w^4}\mathrm{d}w.~~(a\in\mathbb{C}\backslash\{0\})
\]
The regularity, completeness and period conditions
are satisfied when $-a$ is a sufficiently large
positive real number (e.g. $-a>1$). For the proof of regularity,
see Lemma~\ref{lem-a=b}.
\end{example}

\section{Non-orientable stationary surfaces and examples}

In this section we will consider non-orientable algebraic stationary surfaces and show that
the total curvature of them is always greater than $4\pi$.
For this purpose we need to consider their oriented double covering surface $\widetilde{M}$,
and characterize the Weierestrass data over $\widetilde{M}$.
This is a natural extension of Meeks' characterization of
non-orientable minimal surfaces in $\mathbb{R}^3$ \cite{Meeks}.

\subsection{Representation of non-orientable stationary surfaces}

\begin{theorem}\label{thm-nonorientable}
Let $\widetilde{M}$ be a Riemann surface with an anti-holomorphic involution
$I:\widetilde{M}\to \widetilde{M}$ (i.e., a conformal automorphism of
$\widetilde{M}$ reversing the orientation) without fixed points.
Let $\{\phi,\psi,\mathrm{d}h\}$ be a set of Weierstrass data on $\widetilde{M}$
such that
\begin{equation}\label{eq-nonorientable}
\phi\circ I=\bar\psi,~~
\psi\circ I=\bar\phi,~~
I^*\mathrm{d}h=\overline{\mathrm{d}h},
\end{equation}
which satisfy the regularity and period conditions as well.
Then they determine a non-orientable stationary surface
\[
M=\widetilde{M}/\{\mathrm{id},I\}\to \mathbb{R}^4_1
\]
by the Weierstrass representation formula \eqref{x}.

Conversely, any non-orientable stationary surface ${\bf x}:M\to \mathbb{R}^4_1$
could be constructed in this way.
\end{theorem}
\begin{remark}
Geometrically, \eqref{eq-nonorientable} is the consequence of reversing the
orientation of the tangent plane by $z\to \bar{z}$, and reversing the induced
orientation of the normal plane by interchanging the lightlike normal directions
$[{\bf y}],\ [{\bf y}^*]$.
\end{remark}
\begin{proof}[Proof to Theorem \ref{thm-nonorientable}]
We prove the converse first.
It is well-known that any non-orientable surface $M$ has a orientable
two-sheeted covering surface $\widetilde{M}$ with an orientation-reversing
homeomorphism $I$, and $M$ is realized as the quotient surface
\[
M=\widetilde{M}/\mathbb{Z}_2=\widetilde{M}/\{\mathrm{id},I\}.
\]
Denote $\pi$ the quotient map.
Notice that $\widetilde{M}$ is endowed with the complex structure
induced from the metric.
When $z$ is a local complex coordinate over a domain $U\subset\widetilde{M}$
which projects to $M$ one-to-one,
$\bar{z}$ is also a coordinate over $I(U)$ compatible with the orientation on $\widetilde{M}$.

Consider the stationary surface
$\tilde{\bf x}\triangleq {\bf x}\circ \pi:\widetilde{M}\to \mathbb{R}^4_1$.
In the chart $(U,z)$ we have
\[
\tilde{{\bf x}}_z \mathrm{d}z=\Big(\phi+\psi, -\mathrm{i}(\phi-\psi),1-\phi\psi,1+\phi\psi\Big)\mathrm{d}h~.
\]
Then in the corresponding chart $(I(U),w=\bar{z})$,
consider $\tilde{{\bf x}}^*=\tilde{{\bf x}}\circ I:I(U)\to \mathbb{R}^4_1$
and we have
\[
\tilde{{\bf x}}^*_w \mathrm{d}w=\Big(\bar\phi+\bar\psi, \mathrm{i}(\bar\phi-\bar\psi),1-\bar\phi \bar\psi,1+\bar\phi \bar\psi\Big)\mathrm{d}\bar{h}~.
\]
This implies \eqref{eq-nonorientable}.

Now we prove the first part. If $M=\widetilde{M}/\{\mathrm{id},I\}$ as described
in the theorem and $\phi,\psi,\mathrm{d}h$ satisfy condition \eqref{eq-nonorientable}
as well as the regularity and period conditions,
then the integral along any path $\gamma\subset\widetilde{M}$
yields two stationary surfaces
\begin{eqnarray*}
\tilde{\bf x} \!&=& 2~\mathrm{Re}\int_\gamma \Big(\phi+\psi, -\mathrm{i}(\phi-\psi),1-\phi\psi,1+\phi\psi\Big)\mathrm{d}h,\\
\tilde{{\bf x}}\circ I~=~\tilde{{\bf x}}^*\!&=& 2~\mathrm{Re}\int_\gamma\Big(\bar\psi+\bar\phi, -\mathrm{i}(\bar\psi-\bar\phi),1-\bar\psi \bar\phi,1+\bar\psi \bar\phi\Big)\mathrm{d}\bar{h}~.
\end{eqnarray*}
If we assign the same initial value, then after either integration above
we get the same result, because they are the real parts of
a holomorphic vector-valued function and its complex conjugate.
So $p\in \widetilde{M}$ and $I(p)\in\widetilde{M}$ are mapped to
the same point in $\mathbb{R}^4_1$, yet with opposite induced orientations
on the same surface. After taking quotient we get a stationary
immersion of the non-orientable $M$ into $\mathbb{R}^4_1$.
This finishes the proof.
\qed\end{proof}

As an application of this theorem, we give a natural generalization
of Meeks and Oliveira's construction of minimal M\"obius strip.

\begin{example}
[Generalization in $\mathbb{R}^4_1$ of Meeks' minimal M\"obius strip]
\label{exa-meeks}
This is defined on $\widetilde{M}=\mathbb{C}\backslash\{0\}$ with involution $I:z\to -1/\bar{z}$,
and the Weierstrass data be
\begin{equation}
\phi=\frac{z-\lambda}{z-\bar\lambda}\cdot z^{2m},~~
\psi=\frac{1+\bar\lambda z}{1+\lambda z}\cdot\frac{1}{z^{2m}},~~
\mathrm{d}h=\mathrm{i}\frac{(z-\bar\lambda)(1+\lambda z)}{z^2}\mathrm{d}z,
\end{equation}
where $\lambda$ is a complex parameter satisfing
$\lambda\ne \pm 1, |\lambda|=1$, and the integer $m\ge 1$.
\end{example}
\begin{remark}
When $\lambda=\pm\mathrm{i}$ we have $\phi=-1/\psi$, and the example above
is equivalent to Oliveira's examples in $\mathbb{R}^3$
 \cite{Oliveira}.
(Meeks' example \cite{Meeks} corresponds to the case $m=1$.)
Otherwise this is a full map in $\mathbb{R}^4_1$.
Furthermore, for fixed $m$ these examples are not congruent to each other
unless the values of the parameter $\lambda$
are the same or differ by complex conjugation, because the cross ratio
\[\mathrm{cr}(0,\infty;\lambda,\bar\lambda)
=\frac{\lambda}{\bar\lambda}\]
between the zeros and poles in the normal form of $\phi$ is an invariant.
\end{remark}
\begin{proposition}\label{prop-olivaira}
Example~\ref{exa-meeks} is a complete immersed stationary
M\"obius strip with a regular end and total Gaussian curvature $2(2m+1)\pi$.
\end{proposition}
\begin{proof}
We start from a general case, a M\"obius strip
$M=\widetilde{M}/\{\mathrm{id},I\}\to\mathbb{R}^4_1$ with
\[
\widetilde{M}=\mathbb{C}\backslash\{0\},~~I:z\to -1/\bar{z},~~
\phi(z)=\frac{az+b}{cz+d}\cdot z^{2m}.~~~(a,b,c,d\in \mathbb{C}, ad-bc\ne 0)
\]
To satisfy condition \eqref{eq-nonorientable}, there should be
\[
\psi=\overline{\phi(-1/\bar{z})}=
\frac{\bar{b}z-\bar{a}}{\bar{d}z-\bar{c}}\cdot\frac{1}{z^{2m}}~.
\]
The surface is regular outside the ends $\{0,\infty\}$.
Together with $\mathrm{d}h^*=\overline{\mathrm{d}h}$, this implies
\[
\mathrm{d}h=\mathrm{i}\frac{(cz+d)(\bar{d}z-\bar{c})}{z^2}\mathrm{d}z
\]
up to multiplication by a real constant.
Under these conditions it is easy to verify that the metric is complete.

Next, let us check the period conditions. The horizontal periods vanish
automatically since $\phi\mathrm{d}h,\psi\mathrm{d}h$ has no residues
at $0$ and $\infty$. The vertical periods must vanish, hence
$|d|^2=|c|^2,~|b|^2=|a|^2.$ Without loss of generality we may write
\[
\phi=\frac{z-\lambda}{z-\bar\lambda}\cdot z^{2m},~~~|\lambda|=1.
\]
To simplify $\phi$ to this form we have utilized the freedom to change
complex coordinate by $z\to \mu z$ and the (fractional) linear transformation
$\phi\to \mu'\phi$ induced from the Lorentz transformation of $\mathbb{R}^4_1$
(see \eqref{trans}).

We are left to verify $\phi\ne \bar\psi$ over $\mathbb{C}\backslash\{0\}$.
(At the ends $z=0,\infty$ it is obviously true. So they are regular ends.)
Suppose $\phi(z)=\bar\psi(z)$ form some $z\in\mathbb{C}$.
Substitute the expressions of $\phi,\psi$ into it. We obtain
\[
|z|^{4m}=\frac{(z-\bar\lambda)(-1/\bar{z}-\lambda)}
{(z-\lambda)(-1/\bar{z}-\bar\lambda)}
=\mathrm{cr}\left(z,-1/\bar{z};\bar\lambda,\lambda\right).
\]
Since the cross ratio at the right hand side takes a real value,
four points $z,\frac{-1}{\bar{z}};\bar\lambda,\lambda$
are located on a circle $C$ in the complex plane $\mathbb{C}$.

We assert that this circle $C$ could not be identical to the unit circle.
(Otherwise $|z|=1$ and the cross ratio above is $1$. This holds true only if
$z=\frac{-1}{\bar{z}}$, which is impossible, or $\lambda=\bar\lambda=\pm 1$,
which has been ruled out in Example~\ref{exa-meeks}.)

Circle $C$ intersects the unit circle at $\lambda$ and $\bar\lambda$.
Observe that any circle passing through $z,\frac{-1}{\bar{z}}$
will intersect the unit circle at an antipodal point pair.
(Because under the inverse of the standard stereographic projection,
$z,\frac{-1}{\bar{z}}$ correspond to two antipodal points on $S^2$,
and the unit circle corresponds to the equator. Any circle passing through
the inverse images of $z,\frac{-1}{\bar{z}}$ on $S^2$ will intersect
the equator again at two antipodal points. After taking stereographic
projection back to $\mathbb{C}$ we get the conclusion.)
As a consequence, $\lambda=\pm\mathrm{i}$. But this time
the aforementioned cross ratio could only take value as a negative real number
(because on circle $C$, $z,\frac{-1}{\bar{z}}$ must be separated by $\pm\mathrm{i}$).
This contradiction finishes our proof.
\qed\end{proof}

When $m=1$ this example has smallest possible total curvature
$6\pi$ among non-orientable algebraic stationary surfaces.
(Note that the classical Henneberg surface in $\mathbb{R}^3$ has total curvature $2\pi$, yet with four branch points.)
This conclusion is the corollary of a series of propositions
below.

\subsection{Non-orientable stationary surfaces of least total curvature}

In general we are interested in finding least possible total curvature
for non-orientable stationary surfaces of a given topological type.
This is motivated by discussions of F. Martin in \cite{Martin}.
Compared with minimal surfaces in $\mathbb{R}^3$,
this general case looks even more interesting (at least to the authors).

As a consequence of Theorem~\ref{thm-nonorientable}, for a
complete non-orientable stationary surface with double covering
$\widetilde{M}$ of genus $g$ with $2r$ ends, there must be
$\deg\phi=\deg\psi$; the index formula \eqref{eq-deg4}
as well as the Jorge-Meeks formula \eqref{eq-jorgemeeks2} implies
\begin{equation}\label{eq-jorgemeeks3}
-\int_M K=2\pi \Big(\deg\phi-\sum_{j=1}^{r} |\mathrm{ind}_{p_j}|\Big)=2\pi\Big(g+r-1+\sum_{j=1}^{r} \tilde{d}_j\Big)~.
\end{equation}
Because $r\ge 1$ and $\tilde{d}_j\ge 1$, we know
\[-\int_M K=\ge 2\pi(g+1).\]
A better estimation is given in the following proposition.

\begin{proposition}\label{prop-least}
Given a non-orientable surface $M$ whose double covering space \
$\widetilde{M}$ has genus $g$ and finite many punctures,
there does not exist complete algebraic stationary immersion ${\bf x}:M\to\mathbb{R}^4_1$
with total Gaussian curvature $-\int_M K\mathrm{d}M=2\pi(g+1)$.
In other words, under our assumptions there must be
\begin{equation}\label{eq-lowerbound1}
-\int_M K\mathrm{d}M\ge 2\pi(g+2).
\end{equation}
\end{proposition}
\begin{proof}
Consider the lift of ${\bf x}$, i.e.,
$\tilde{\bf x}:\widetilde{M}\to\mathbb{R}^4_1$.
Since the immersion is algebraic and
$-\int_{\widetilde{M}} K<+\infty$,
it has finite many regular or good singular ends,
and the total number is a even number $2r$
($r$ is the number of ends of $M$).
By the modified Jorge-Meeks formula \eqref{eq-jorgemeeks3}
and $r\ge 1$, a Chern-Osserman type inequality is obtained:
\[
-\int_M K\mathrm{d}M\ge 2\pi(g+1).
\]
Suppose the equality is achieved.
Then there must be two ends for $\widetilde{M}$ and $\tilde{d}_1=\tilde{d}_2=1$.
Both of them are regular ends or good singular ends at the same time.
We will show that in either case there will be a contradiction.

\textbf{Case 1: regular end(s)}. The multiplicity $\tilde{d}_1=d_1=1$,
and $\tilde{\bf x}_z\mathrm{d}z$ for the end $p_1$ has a pole of order $2$.
In a local coordinate chart with $z(p_1)=0$ we write out the Laurent expansion of $\tilde{\bf x}_z$:
\[
\tilde{\bf x}_z=\frac{1}{z^2}~{\bf v}_2+\frac{1}{z}~{\bf v}_1+(\text{holomorphic part}).
\]
Since this is a regular end, ${\bf v}_2$ is an isotropic vector whose
real and imaginary parts span a 2-dimensional spacelike subspace.
${\bf v}_1$ is a real vector orthogonal to ${\bf v}_2$ by the period condition
and $<\tilde{\bf x}_z,\tilde{\bf x}_z>=0$.
Thus in $\mathbb{R}^4_1$ there exist a constant non-zero real vector
${\bf v}_0\perp {\bf v}_2,{\bf v}_1$.

At the other end $p_2=I(p_1)$ with local coordinate $w=\bar{z}$,
because $\tilde{\bf x}_w=\tilde{\bf x}_{\bar{z}}$, we know the same ${\bf v}_0$
is orthogonal to the principal part of the Laurent series.
Thus $<\tilde{\bf x}_z\mathrm{d}z,{\bf v}_0>$ is a holomorphic
$1$-form, and $<\tilde{\bf x},{\bf v}_0>$ is a harmonic function defined on the
whole compact Riemann surface. It must be a constant; hence $\tilde{\bf x}$
as well as ${\bf x}$ is contained in a 3-dimensional subspace of $\mathbb{R}^4_1$.

Yet this is impossible. Since in $\mathbb{R}^3_1$ or $\mathbb{R}^3_0$
there exist no immersed spacelike non-oriented surfaces. The possibility of
$M\subset\mathbb{R}^3$ could be ruled out by Schoen's famous result \cite{Schoen}
that any complete, connected, oriented minimal surface in $\mathbb{R}^3$
with two embedded ends is congruent to the catenoid.
(Alternatively, we may argue by the maximal principle once again.
Since the unique end of $M$ is an embedded end in $\mathbb{R}^3$,
which is either a catenoid end or a planar end,
one can choose the coordinate of $\mathbb{R}^3$ suitably so that the height
function ${\bf x}_3$ is bounded from below over the whole $M$.
Such a harmonic function must be a constant, and $M\subset \mathbb{R}^2$.
Contradiction.)

\textbf{Case 2: good singular end(s)}.
At the good singular end $p_1$, without loss of generality,
suppose it has $\mathrm{ind}=m\ge 1$ and $\phi(p_1)=\psi(p_1)=0$.
Then $\tilde{d}_1=d_1-m=1$,
and $\tilde{\bf x}_z\mathrm{d}z$ has a pole of order $m+2$ at $p_1$.
There always exists a suitable local coordinate $z$ such that $z(p_1)=0$ and
\[
\mathrm{d}h=\frac{\mathrm{d}z}{z^{m+2}},~~
\phi(z)=a_0 z^m + a_1 z^{m+1} + O(z^{m+2}),~~
\psi(z)=b_1 z^{m+1} + O(z^{m+2}).
\]
By \eqref{x} we know
\begin{eqnarray*}
\tilde{\bf x}_z \mathrm{d}z &=&
\Big(\phi+\psi, -\mathrm{i}(\phi-\psi),1-\phi\psi,1+\phi\psi\Big)\mathrm{d}h\\
&=&\frac{\mathrm{d}z}{z^{m+2}}~\begin{pmatrix}0 \\0 \\1 \\1\end{pmatrix}
+\frac{\mathrm{d}z}{z^2}~\begin{pmatrix}a_0 \\-\mathrm{i} a_0 \\0 \\0\end{pmatrix}
+\frac{\mathrm{d}z}{z}~\begin{pmatrix}a_1+b_1 \\-\mathrm{i} (a_1-b_1) \\0 \\0\end{pmatrix}
+(\text{holomorphic part}).
\end{eqnarray*}
Take ${\bf v}_0=(0,0,1,1)$. We can argue as in case 1 to show that
$<\tilde{\bf x},{\bf v}_0>$ is a harmonic function defined on the
whole compact Riemann surface, hence be a constant. (The key point is that $M$ has only one end.)
Thus $\tilde{\bf x}(\widetilde{M})$ as well as $x(M)$ is contained
in an affine space $\mathbb{R}^3_0$ (orthogonal to ${\bf v}_0$).
Yet this is also impossible for a non-oriented spacelike surface.
\qed\end{proof}

We will show that the lower bound could be improved to be $2\pi(g+3)$, the same as the case for non-orientable minimal surfaces in $\mathbb{R}^3$.

\begin{theorem}\label{thm-least}
Given a non-orientable surface $M$ whose double covering space \
$\widetilde{M}$ has genus $g$ and finite many punctures,
there does not exist complete algebraic stationary immersion ${\bf x}:M\to\mathbb{R}^4_1$
with total Gaussian curvature $-\int_M K\mathrm{d}M=2\pi(g+2)$.
In other words, under our assumptions there must be
\begin{equation}\label{eq-lowerbound2}
-\int_M K\mathrm{d}M\ge 2\pi(g+3).
\end{equation}
\end{theorem}
\begin{proof}
As in the proof to Proposition~\ref{prop-least}, consider the lift of ${\bf x}$, i.e.,
$\tilde{\bf x}:\widetilde{M}\to\mathbb{R}^4_1$.
Suppose the lower bound $2\pi(g+2)$ is attained.
Then $\widetilde{M}$ has two ends $p_1,p_2$ with
 $\tilde{d}_1=\tilde{d}_2=2$ by \eqref{eq-jorgemeeks3}.
By symmetry, both of them are regular or good singular ends at the same time. Each possibility is ruled out using different arguments.

When both ends are good singular ends, we use the same argument as in Case~2 of Proposition~\ref{prop-least}.
At the good singular end $p_1$, without loss of generality,
suppose it has $\mathrm{ind}=m\ge 1$ and $\phi(p_1)=\psi(p_1)=0$.
Then $\tilde{d}_1=d_1-m=2$,
and $\tilde{\bf x}_z\mathrm{d}z$ has a pole of order $m+3$ at $p_1$.
There always exists a suitable local coordinate $z$ such that $z(p_1)=0$ and
\[
\mathrm{d}h=\frac{\mathrm{d}z}{z^{m+3}},~~
\phi(z)=a_0 z^m + a_1 z^{m+1} + O(z^{m+2}),~~
\psi(z)=b_1 z^{m+1} + O(z^{m+2}).
\]
By the Weierstrass representation formula we know
\begin{eqnarray*}
\tilde{\bf x}_z
&=&\frac{1}{z^{m+3}}~\begin{pmatrix}0 \\0 \\1 \\1\end{pmatrix}
+\frac{1}{z^3}~\begin{pmatrix}a_0 \\-\mathrm{i} a_0 \\0 \\0\end{pmatrix}
+\frac{1}{z^2}~\begin{pmatrix}a_1+b_1 \\-\mathrm{i} (a_1-b_1) \\0 \\0\end{pmatrix}
+\frac{1}{z}~\begin{pmatrix}a_2+b_2 \\-\mathrm{i} (a_2-b_2) \\a_0 b_1 z^{m-1} \\-a_0 b_1z^{m-1}\end{pmatrix}\\
&& +(\text{holomorphic part}).
\end{eqnarray*}
Take ${\bf v}_0=(0,0,1,1)$. Because $m\ge 1$,
$<\tilde{\bf x},{\bf v}_0>$ is a harmonic function (with the leading term $\ln|z|$) bounded from
below or above in a neighborhood of $p_1$.
Since $M$ has only one end around which the assertion above is
still valid, we conclude that $<\tilde{\bf x},{\bf v}_0>$ is
a harmonic function bounded from
below or above over the whole compactified surface.
It must be a constant, and the surface is contained in a 3-space. As in Proposition~\ref{prop-least} this leads to a contradiction.

In case that both ends are regular, consider the anti-holomorphic automorphism $I:\widetilde{M}\to \widetilde{M}$ without fixed points and $M\cong \widetilde{M}/\{I\}$
Under the assumptions above, the Gauss map $\phi$
could be viewed as a continuous map from the oriented double
 covering space to the round 2-sphere such that
 \[
 \phi:\{\widetilde{M};I\}\to S^2\subset \mathbb{R}^3,~~~\text{s.t.},~\phi(p)\ne \phi(I(p)).
 \]
 It is a standard fact that such a map is homotopic to an odd map
 \[
 \tilde\phi\triangleq \frac{\phi(p)-\phi(I(p))}{|\phi(p)-\phi(I(p))|}, ~~~\tilde\phi(I(p))=-\tilde\phi(p),
 \]
where we give the homotopy $H(p,t)= \frac{\phi(p)-t\phi(I(p))}{|\phi(p)-t\phi(I(p))|}$ directly.
According to Theorem~\ref{thm-odd} in Appendix~B,
the mapping degree of such an odd map and $g-1$ must be
both even or both odd. Thus the mapping degree could not
be $g+2$. This finishes the proof.
\qed\end{proof}

\begin{conjecture}\label{conj}
The lower bound \eqref{eq-lowerbound2}
$-\int_M K\mathrm{d}M\ge 2\pi(g+3)$
is sharp for any given $g\ge 0$. In other words, there always exists an complete, immersed, algebraic non-orientable stationary surfaces whose double covering surface has genus $g$.
\end{conjecture}

This is a generalization of conjecture~1 in \cite{Martin} for non-orientable minimal surfaces in $\mathbb{R}^3$.
It is verified in $\mathbb{R}^3$ when $g=0$ and $g=1$. The corresponding examples are the Meeks' M\"obius strip \cite{Meeks} and Lopez's Klein bottle \cite{Lopez}.
For higher genus $g$ this conjecture is still open.

As the direct consequence of Theorem~\ref{thm-least}
we obtain the following result:

\begin{theorem}\label{thm-nonorientable4pi}
There does not exist a complete, algebraic, immersed
non-orientable stationary surface in $\mathbb{R}^4_1$
with total Gaussian curvature $-\int_M K\mathrm{d}M=4\pi$.
\end{theorem}

Combined with Theorem~\ref{thm-4pi1}, this
 finishes the proof to our classification theorem (Theorem A in the Introduction).

\begin{remark}\label{rem-oliveira}
We note a significant difference between non-orientable stationary surfaces in $\mathbb{R}^4_1$ and $\mathbb{R}^4$.
Oliveira \cite{Oliveira} constructed complete M\"obius band in $\mathbb{R}^4$ with total curvature $2\pi m$ for any $m\ge 2$.
So the total curvature $4\pi$ could be realized in that case.
\end{remark}

In the proof to Theorem~\ref{thm-least}, when treating the special case with only regular ends, indeed we have obtained the following proposition,
which is a partial generalization of Meeks' result (Corollary~1 in \cite{Meeks}):
\begin{proposition}
A complete non-orientable stationary surface in $\mathbb{R}^4_1$
of algebraic type without singular ends must have total curvature $-\int_M K\mathrm{d}M=2\pi m$, where $m\equiv g-1(\mathrm{mod}~2)$, and $g$ is the genus of the oriented double covering surface.
\end{proposition}
So far we do not know whether it is true in the general case
when good singular ends exist.

\section{Non-algebraic examples with small total Gaussian curvature}
Recall the following classical result.
\begin{theorem} Let $(M,\mathrm{d}s^2)$ be a non-compact surface with a complete metric. Suppose $\int_M|K|\mathrm{d}M<+\infty$, then:

(1) (Huber\cite{Huber}) There is a compact Riemann surface
$\overline{M}$ such that $M$ as a Riemann surface is biholomorphic to $\overline{M}\backslash\{p_1,p_2,\cdots,p_r\}$.

(2) (Osserman\cite{Osser}) When this is a minimal surface in $\mathbb{R}^3$ with the induced metric $\mathrm{d}s^2$, the Gauss map $G=\phi=-1/\psi$ and the
height differential $\mathrm{d}h$ extend to each end $p_j$ analytically.

(3) (Jorge and Meeks \cite{Jor-Meeks}) As in (1) and (2), suppose minimal surface $M\to \mathbb{R}^3$ has $r$ ends and $\overline{M}$ is the compactification with genus $g$. The total curvature is related with these topological invariants via the Jorge-Meeks formula:
\begin{equation}\label{eq-jorgemeeks}
\int_M
K\mathrm{d}M=2\pi\left(2-2g-r-\sum_{j=1}^r d_j\right)~,
\end{equation}
Here $d_j+1$ equals to the highest order of the pole of ${\bf x}_z \mathrm{d}z$ at $p_j$,
and $d_j$ is called \emph{the multiplicity at the end $p_j$}.
\end{theorem}
Huber's conclusion (1) means \emph{finite total curvature $\Rightarrow$ finite topology}, which is a purely intrinsic result. In particular, this is valid also for stationary surfaces in $\mathbb{R}^4_1$.
Surprisingly, as to the extrinsic geometry, Osserman's result
$2)$ is no longer true in $\mathbb{R}^4_1$. In particular we
have non-algebraic counter-examples given below:

\begin{example}
[$M_{k,a}$ with essential singularities and finite total curvature \cite{Ma}]
\label{exa-essen}
\begin{equation}\label{ex-essential}
M_{k,a}\cong \mathbb{C}-\{0\},~\phi=z^k \mathrm{e}^{az},~\psi=-\frac{\mathrm{e}^{az}}{z^k},~\mathrm{d}h=\mathrm{e}^{-az}\mathrm{d}z~.
\end{equation}
where integer $k$ and real number $a$ satisfy
$k\ge 2, 0< a<\frac{\pi}{2}$.
\end{example}
\begin{proposition}\cite{Ma}\label{prop-essential}
Stationary surfaces $M_{k,a}$ in Examples~\ref{exa-essen} are regular, complete stationary surfaces with two ends at $z=0,\infty$ satisfying the period conditions. Moreover their total curvature converges absolutely with
\begin{equation}\label{eq-essential}
\int_M
K\mathrm{d}M=-4\pi k~, ~~\int_M K^{\perp}\mathrm{d}M =0~.
\end{equation}
\end{proposition}

\begin{remark}
Taking different height differential $\mathrm{d}h$ in Example~\ref{exa-essen}, we can obtain
other examples with the same total Gaussian curvature. Yet the total Gaussian curvature $-\int_M K\mathrm{d}M=4\pi$ could not be realized since when $k=1$ the integral is not absolutely convergent.
\end{remark}

Similar to the construction of Example~\ref{exa-essen} and the proof to
Proposition~\ref{prop-essential} as in \cite{Ma}, we have non-oriented, non-algebraic examples as below.

\begin{example}
[stationary M\"obius strips with essential singularities and finite total curvature]
\label{exa-essen2}
\begin{equation}\label{ex-essential2}
\phi=z^{2p-1}\mathrm{e}^{\frac{1}{2}(z-\frac{1}{z})},~~
\psi=\frac{-1}{z^{2p-1}}\mathrm{e}^{\frac{1}{2}(z-\frac{1}{z})},~~
\mathrm{d}h=\mathrm{d}~\mathrm{e}^{-\frac{1}{2}(z-\frac{1}{z})}.~~(p\in\mathbb{Z}_{\ge 2})
\end{equation}
\end{example}
\begin{proposition}
Example~\ref{exa-essen2} is a complete immersed stationary M\"obius strip with finite total curvature
$\int|-K+\mathrm{i}K^{\perp}|\mathrm{d}M<+\infty$. We have
\[
-\int_M K\mathrm{d}M=2(2p-1)\pi,~~\int_M K^{\perp}\mathrm{d}M=0.
\]
In particular, the smallest possible value of their total Gaussian curvature is $6\pi$.
\end{proposition}
\begin{proof}
It is easy to verify $\phi^*=\bar\psi,\psi^*=\bar\phi,
\mathrm{d}h^*=\overline{\mathrm{d}h}$. So we obtain a M\"obius strip
according to Theorem~\ref{thm-nonorientable}.

The regularity is also easy to verify.
For example, if there exist $z$ such that $\phi(z)=\bar\psi(z)$, by \eqref{ex-essential2}
we get $z\bar{z}~\mathrm{e}^{\mathrm{i}\cdot\mathrm{Im}(z-\frac{1}{z})}=-1$. So $|z|=1$
and $z=\mathrm{e}^{\mathrm{i}\theta}$. Insert this back to the previous equation; we obtain
$\mathrm{e}^{2\mathrm{i}\sin\theta}=-1$, which is impossible.

Next we check the period condition. Since
\[
\mathrm{d}h=\mathrm{d}~\mathrm{e}^{-\frac{1}{2}(z-\frac{1}{z})},~~
\phi\psi\mathrm{d}h=\mathrm{d}~\mathrm{e}^{\frac{1}{2}(z-\frac{1}{z})},
\]
both being exact $1$-forms, there are no vertical periods. At the same time,
\[
\phi\mathrm{d}h=-\frac{1}{2}\left(1+\frac{1}{z^2}\right)z^{2p-1}\mathrm{d}z,~~
\psi\mathrm{d}h=-\frac{1}{2}\left(1+\frac{1}{z^2}\right)\frac{-1}{z^{2p-1}}\mathrm{d}z.
\]
When $p\ge 2$ neither of these $1$-forms has residue. So there are no horizontal periods.

By direct computation one can show that the integral of the absolute total curvature
$\int|-K+\mathrm{i}K^{\perp}|\mathrm{d}M$ is asymptotic to
$\int |z|^{2-4p}\mathrm{d}z\mathrm{d}\bar{z}$ when $z\to\infty$, or to
$\int |z|^{4p-6}\mathrm{d}z\mathrm{d}\bar{z}$ when $z\to 0$.
Thus when $p\ge 2$ the total curvature integral
converges absolutely. Approximate $\widetilde{M}$ by domains
$A_{r,R}\triangleq\{0<r\le |z|\le R\}$. By Stokes theorem we get
\begin{align*}
\int_{A_{r,R}}(-K+\mathrm{i}K^{\perp})\mathrm{d}M
&=2\mathrm{i}\int_{A_{r,R}}
\frac{\phi_z\bar{\psi}_{\bar{z}}}{(\phi-\bar{\psi})^2}
\mathrm{d}z\wedge \mathrm{d}\bar{z}\\
&=-2\mathrm{i}\oint_{|z|=R}
\frac{\phi_z}{\phi-\bar{\psi}}\mathrm{d}z+2\mathrm{i}\oint_{|z|=r}
\frac{\phi_z}{\phi-\bar{\psi}}\mathrm{d}z~,
\end{align*}
where
\[\frac{\phi_z}{\phi-\bar{\psi}}=
\frac{|z|^{4p-2}\left(\frac{1}{2}+\frac{2p-1}{z}+\frac{1}{2z^2}\right)}
{|z|^{4p-2}+\mathrm{e}^{-\mathrm{i}\cdot\mathrm{Im}(z-\frac{1}{z})}}~.
\]
When $R\to\infty$ the first contour integral converges to $-2\mathrm{i}(2p-1)\cdot 2\pi\mathrm{i}$.
When $r\to 0$ the second contour integral converges to $0$. This completes the proof.
\qed\end{proof}

In the discussion above, when $p=1$ the horizontal period condition
is violated, and the total curvature integral does not
converge absolutely. Thus among these simplest examples
(including Examples~\ref{exa-essen}) we can not find
one with total Gaussian curvature $4\pi$. This motivates
the following

\begin{conjecture}
There does NOT exist any complete, non-algebraic stationary
surfaces immersed in $\mathbb{R}^4_1$ with finite total Gaussian
curvature $-\int K\mathrm{d}M=4\pi$.
\end{conjecture}

\section{Appendix A}

In the proof to our main theorem, a key fact is that a complete,
algebraic stationary surface $M\cong \mathbb{C}\backslash\{0\}$
with two good singular ends and $-\int K\mathrm{d}M=4\pi$ must have
other singular (branch) points. This follows from
\begin{lemma}\label{lem-main}
For any positive integer $m\in \mathbb{Z}^+$ and any non-zero
 complex parameters $a,b\in\mathbb{C}\backslash\{0\}$ whose sum
$a+b=-e^{it}$ is a given unit complex number ($t\in\mathbb{R}$),
there always exists a solution $z\ne 0$ to the equation
\begin{equation}\label{eq-main}
(\bar{z}-\bar{a})(z-b)=\frac{z^{m+1}}{\bar{z}^m}.
\end{equation}
\end{lemma}
\begin{proof}
By change of coordinates $z\to ze^{it}$, we may consider an
equivalent equation
\begin{equation}\label{eq-main2}
(\bar{z}-\bar{a})(z-b)=\lambda\frac{z^{m+1}}{\bar{z}^m},
\end{equation}
where $\lambda=e^{it(2m+1)}$ is a unit complex number, and $a,b$ are nonzero complex parameters satisfying
\[
a+b=-1.
\]
In other words, the middle point of the segment
$\overline{ab}$ is $-1/2$. We will show that there always
exists a solution $z\ne 0$ to the equation \eqref{eq-main2}
for any given unit complex number $\lambda$ and when $a+b=-1$.

First let us explain the basic idea of our proof.
Consider the equal-module locus \[
\Gamma=\{z\in\mathbb{C}:|z-a||z-b|=|z|\}
\]
where the two sides of \eqref{eq-main2} have equal modules.
Since $|z-a||z-b|<|z|$ when $z=a,b$,
and $|z-a|\cdot|z-b|>|z|$ when $z$ is big enough, by continuity
we know this locus is non-empty. It is easy to see that
$\Gamma=\cup\Gamma_j$ is a union of several continuous,
connected, (simple) closed curves.
Next we compare the argument of the
complex functions at both sides of \eqref{eq-main2}. Define
\[
\arg_L\triangleq \arg[(\bar{z}-\bar{a})(z-b)],~~
\arg_R\triangleq \arg(z^{m+1}/\bar{z}^m)=(2m+1)\arg(z).
\]
Note that the two arguments $\arg_L,\arg_R$ can be defined and
extended continuously along any continuous path (without self-intersection) on $\mathbb{C}$.
We want to find one component $\Gamma_j\subset\Gamma$ such that
\[
\delta= \arg_R-\arg_L,
\]
the difference of the two arguments, will have a
bounded variation greater than $2\pi$. Again by continuity
we know that along $\Gamma_j$ there is some point $z$ at which
both sides of \eqref{eq-main2} share equal modules and
arguments. This will finish our proof.\\

To estimate the variation of the argument difference $\delta$,
the next key point is to construct two points on the locus
$\Gamma$ using some elementary geometry. Now we have to consider
two cases separately.\\

\textbf{Case 1: $|a-b|\le 1$.}

In this case, the complex numbers $a,b$ correspond to two points
located inside the circle $|z+\frac 12|=\frac 12$ and
being symmetric about $z=-\frac 12$.

\begin{figure}
\includegraphics[width=12cm]{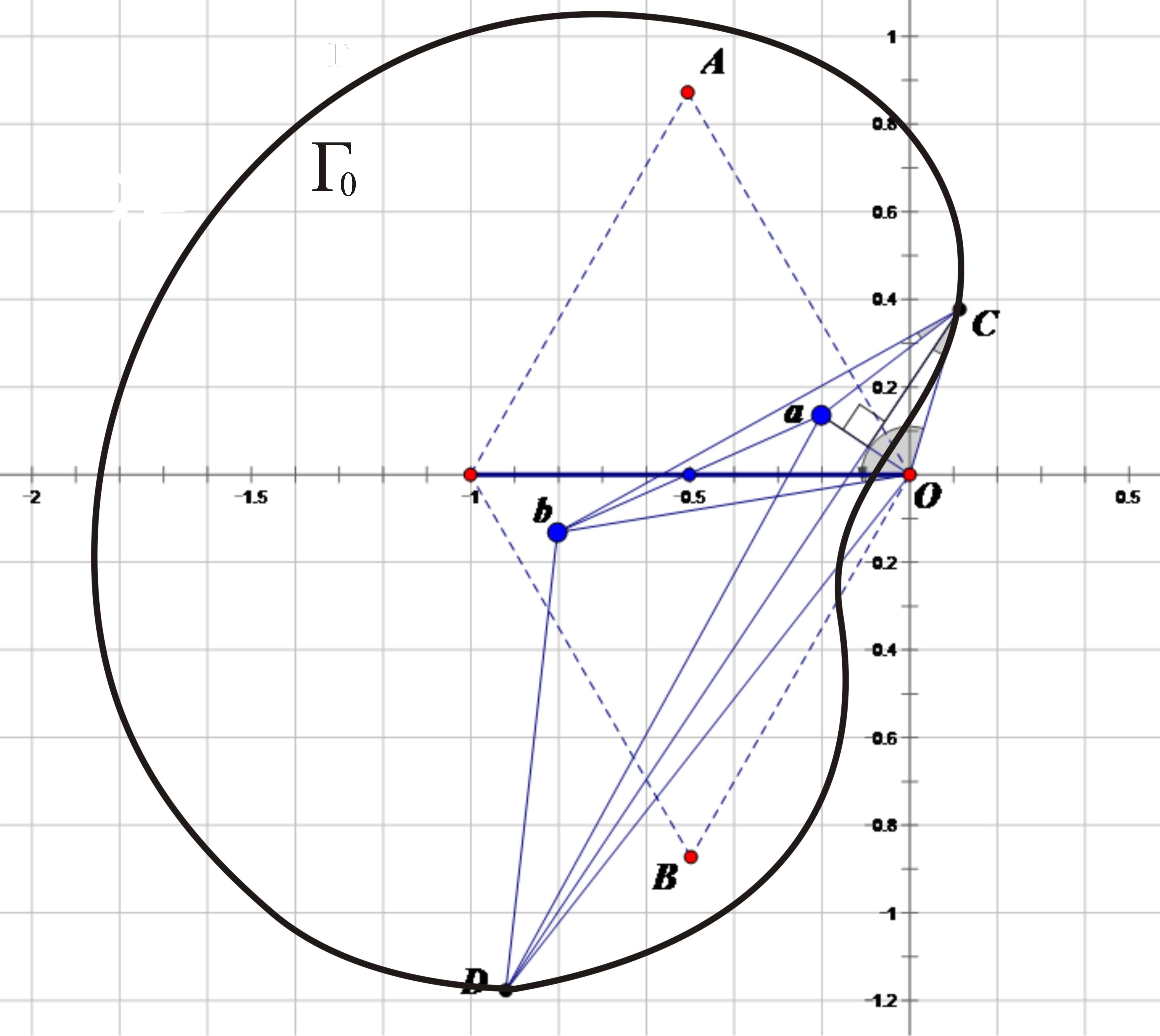}
\caption{Case 1, $|a-b|\le 1$}
\label{fig:case1}
\end{figure}

Suppose $\mathrm{Im}(a)\ge 0, \mathrm{Im}(b)\le 0$.
We want to find two points $C,D$ on $\Gamma$
which are equidistant to $a$ and the origin $O$,
at the same time whose distance to $b$ is $1$ (see Figure~ \ref{fig:case1}).
Such points are exactly the intersection between the bisector
of the segment $\overline{aO}$ and the unit circle centered at $b$.
Because the length $|\overline{ba}|$ and $|\overline{bO}|$ are
no more than $1$ ($|b|\le|b+1/2|+|-1/2|\le 1$), we know the
intersection points $C,D$ exist, and they are distinct.
Let $C$ be the one on the upper half plane.

We claim that $C,D$ must be located on one and the same
component (a simple closed curve) $\Gamma_0\subset\Gamma$.
Notice that the open segment $\overline{CD}$ (on the bisector) is contained in
the interior of the unit circle centered at $b$, hence also in the interior of
\[\Omega=\{z\in\mathbb{C}:|z-a||z-b|<|z|\}.\]
Let $\Gamma_0\subset\Gamma=\partial\Omega$ be the component
passing through $C$.
Then the straight line $CD$ must have at least one more
intersection with $\Gamma_0$, whose coordinate $z$ satisfies $|z-a||z-b|=|z|$
and $|z-a|=|z|$, hence $|z-b|=1$. It has to be $D$ as defined
above. This verifies our claim.

The main consequence of the condition $|a-b|\le 1$ is that
$|b|\le 1$. Using the relation that greater angle is opposite
greater side in the triangle $\triangle bCO$,
 we know
\[2\angle OCD+\angle aCb=\angle OCb \le \angle COb.\]
Similarly, in the triangle $\triangle bDO$ we have
\[2\angle ODC+\angle aDb=\angle ODb \le \angle DOb.\]
Taking sum of these two equalities and using
$\angle COb+\angle DOb=\pi-\angle OCD-\angle ODC$ in the triangle $\triangle OCD$, we get
\[\angle aCb+\angle aDb\le 3(\angle COb+\angle DOb)-2\pi
\le (2m+1)(\angle COb+\angle DOb)-2\pi. \]
Notice that
\begin{equation*}
\begin{split}
\angle aCb=-\arg_L(C)=\left.[\arg(z-a)-\arg(z-b)]\right|_{z=C},\\
\angle aDb=\arg_L(D)=\left.[\arg(z-b)-\arg(z-a)]\right|_{z=D},\\
(2m+1)(\angle COb+\angle DOb)=\arg_R(D)-\arg_R(C).
\end{split}
\end{equation*}
Then the previous inequality amounts to say
\[\delta(D)-\delta(C)\ge 2\pi.\]
Thus along the continuous path connecting $C,D$ which is
part of the equal-module locus $\Gamma_0\subset\Gamma$,
the quotient between $(\bar{z}-\bar{a})(z-b)$ and
$z^{m+1}/\bar{z}^m$ can take any given unit
complex parameter $\lambda$. This finishes the proof in the first case.

We observe that if $\mathrm{Im}(a)\le 0, \mathrm{Im}(b)\ge 0$
(or just interchange $a,b$ in Figure~\ref{fig:case1} above),
the proof is similar.

It seems that our proof relies on the special case of
Figure~\ref{fig:case1} where $a$ is
inside the triangle $bCO$. Indeed, because the positivity of
these two angles was never used in that proof, when $a$ is
outside the triangle $bCO$ the proof is still valid.\\

\textbf{Case 2: $|a-b|> 1$.}

\begin{figure}
\includegraphics[width=12cm]{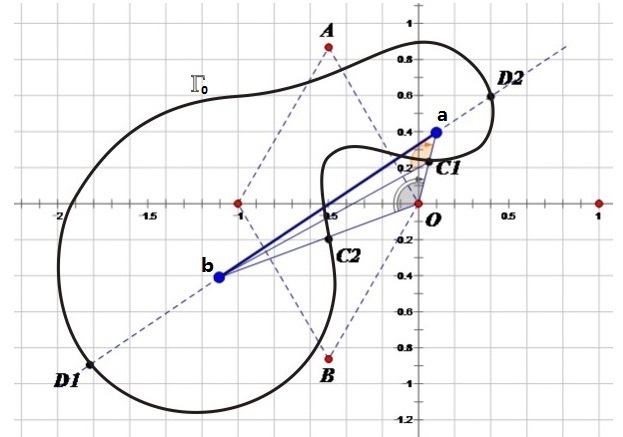}
\caption{Subcase 2.1}
\label{fig:case21}
\end{figure}
Other than case 1, now we have to find a different way to
construct such two points $C,D$ on the equal-module locus.

Consider the triangle $\triangle Oab$. The length of the
median on the side $\overline{ab}$ is less than half of
$|\overline{ab}|$. So $\angle aOb > \pi/2$.
Moreover, any point $C$ on the line segment $Oa$ or $Ob$ will
span an obtuse angle $\angle aCb$.
This is the main consequence of $|a-b|> 1$.

Let $C$ be a moving point on the line segment $Oa$ with
coordinate $z$. Since $|z-a||z-b|<|z|$ when $z$ is very close to
$a$, and $|z-a||z-b|>|z|$ when $z$ is very close to $0$, there
exist at least one intersection between $\overline{Oa}$ and
the equal-module locus $\Gamma$. We take $C_1$ to be the
one closest to $a$ among all such intersection points.
Similarly, we take $C_2$ to be the one closest to $b$
among all intersection points between $\overline{Ob}$ and $\Gamma$.

On the straight line $ab$ we can also find two intersection
points with $\Gamma$, denoted as $D_1, D_2$, such that $D_1, b, a, D_2$ are located on the line $ab$ in the usual
linear order, and $D_1$ ($D_2$) is the closest one among all
intersection points between $\Gamma$ and the ray $aD_1$ ($aD_2$).

Assume that $a$ is on the upper half plane and $b$ is on
the lower half plane. (The other possibilities will be treated
later.) We consider two subcases.

The first subcase is that $C_1, D_1$ are on the same
connected component $\Gamma_0$ of $\Gamma$. Let us start
from $C_1$ and end up with $D_1$ while turning counter-clockwise
around $a$ along $\Gamma_0$. Then it is easy to see that the
variation of $\delta$ will be more than $2\pi$ because the increase of $arg_L$ is
$\angle aC_1b>\frac{\pi}{2}$ and the decrease of $arg_R$ is $(2m+1)\angle D_1OC_1>\frac{(2m+1)\pi}{2}$. As explained before this finishes the proof.

The second subcase is that $D_1$ does not locate on the same
connected component $\Gamma_0$ passing through $C_1$.
Then $b$ must not be contained in the area bounded by
$\Gamma_0$ by our construction and assumption on $D_1$.
Going along $\Gamma_0$ counter-clockwise, $\arg_L$ will decrease
$2\pi$ while $\arg_R$ return to the same initial value.
This shows that the difference of arguments $\delta$ is exactly
$2\pi$, which also finishes our proof by the same reason.

If $a$ is on the lower half plane and $b$ is on the upper half
plane, we may choose the points $C_2, D_2$ instead,
and the proof is the same.

When $a,b\in \mathbb{R}$ and one of them is positive,
the triangle $\triangle Oab$ degenerates. Yet our proof is
still valid without any essential modification.
\qed\end{proof}

\begin{remark}
Before finding the traditional proof to Lemma~\ref{lem-main}
as given above, we seek help from symbolic and
numerical computations. Our colleague Professor Bican Xia,
using an algorithm developed by him \cite{Xia},
succeeded in verifying the conclusion of Lemma~\ref{lem-main}.
This was very important to make us believe the conclusions of
Lemma~\ref{lem-main} and Theorem~\ref{thm-4pi1}, and to motivate
us to find the proof given above.
Professor Xia's method has been utilized in Maple
(version 13 and later).
\end{remark}

Compared to the previous situation where one obtains existence
result, if the parameters are subject to different restrictions,
one can prove non-existence result as below.
This is used in Section~3 to verify the regularity of
Example~\ref{exa-singular1}.

\begin{lemma}\label{lem-a=b}
In Lemma~\ref{lem-main}, if we assume $a=b\in\mathbb{R}$ and $m=1$,
but drop the requirement of $a+b=-1$, then equation~\eqref{eq-main}
has no solutions $z$ when $-a$ is a sufficiently large
positive real number (e.g. $-a>1$)
\end{lemma}
\begin{proof}
Under our assumptions, \eqref{eq-main} simplifies to
\begin{equation}\label{eq-a=b}
|z-a|^2=z^3/|z|^2.
\end{equation}
So $z=r\omega^j$ for some $j\in\{0,1,2\}$ and
$r>0, \omega=\mathrm{e}^{2\pi\mathrm{i}/3}$.
That means on the complex plane $\mathbb{C}$, the solution $z$, if it exists,
must be located on the union of three radial lines.
So we need only to compare $|z-a|^2$ and $|z|$,
the modules at either sides of \eqref{eq-a=b} for $z$ in this subset.
When $r>0$ is small enough or big enough, the module $|z-a|^2$ is obviously larger than $|z|$. Thus intuitively we know that for suitable $a$ there will always be $|z-a|^2>|z|$. It is easy to rigorously verify this assertion;
see the elementary and standard proof in \cite{Ma} (the end of Section~7).
This shows the non-existence of solution $z$.
\qed\end{proof}

\section{Appendix B}

The theorem below is the key lemma in
our proof to Theorem~\ref{thm-least} which shows
the non-existence of a complete, non-oriented, algebraic
stationary surface in $\mathbb{R}^4_1$
with total Gaussian curvature $2\pi(g+2)$
and without any singular points or singular ends.
\begin{theorem}\label{thm-odd}
Let $\widetilde{M}$ be a closed oriented surface of genus $g$,
$I:\widetilde{M}\to\widetilde{M}$ be an
orientation-reversing involution of $\widetilde{M}$ without fixed points.
$\tilde{\phi}:\widetilde{M}\to S^2\subset \mathbb{R}^3$
is an odd map, i.e., $\tilde\phi(I(p))=-\tilde\phi(p)$.
Then $\deg\tilde\phi\equiv g-1 (\mathrm{mod}~2)$
\end{theorem}
The statement reminds us of the famous theorem that any odd
map $f:S^n\to S^n$ has odd degree, which implies the
Borsuk-Ulam Theorem.
We believe that this generalization
is not a new result. Yet to the best of
our knowledge we could not find a reference. The proof below is
provided by Professor Fan Ding from Peking University.
\begin{proof}
Let $M=\widetilde{M}/\{p\sim I(p)\}$ be the quotient surface which is
non-orientable.
$\tilde{\phi}:\widetilde{M}\to S^2$ induces a quotient map $\phi$
from $M$ to the projective plane $\mathbb{R}P^2=S^2/\{x\sim -x\}.$

Decompose $M$ as the connected sum of $g+1$
projective planes $M=M_1\sharp \cdots\sharp M_{g+1}$.
For any $M_j(j=1,\cdots,g+1)$, we choose a closed path $\gamma_j$ in $M_j$
representing the generator of the first homology group
$H_1(M_j,\mathbb{Z}_2)$, which lifts to a path
$\tilde\gamma_j\subset\widetilde{M}$
whose end points are a pair of antipodal points.
As an odd map, $\tilde{\phi}$ maps $\tilde\gamma_j$ to another
path connecting antipodal points, which projects to
a closed path representing the generator of $H_1(\mathbb{R}P^2,\mathbb{Z}_2)$.
Thus the induced map
$$\phi^*:H^1(\mathbb{R}P^2;\mathbb{Z}_2)\to H^1(M;\mathbb{Z}_2)$$ on
the cohomology groups is given by $$\phi^*(\alpha)=\alpha_1+\cdots
+\alpha_{g+1},$$ where $0\neq \alpha\in
H^1(\mathbb{R}P^2;\mathbb{Z}_2)$, and $\alpha_j\in
H^1(M;\mathbb{Z}_2)$($j=1,\ldots,g+1$) satisfies
$\alpha_j([\gamma_i])=0$ for $i\neq j$ and
$\alpha_j([\gamma_j])=1$. Since the intersection between the
homology classes $[\gamma_i]\in H_1(M;\mathbb{Z}_2)$ and
$[\gamma_j]\in H_1(M;\mathbb{Z}_2)$ is $1$ when $i=j$ and $0$ when
$i\neq j$, the Poincar\'{e} dual of $\alpha_j$ is $[\gamma_j]$.
Thus the Poincar\'{e} dual of $\phi^*(\alpha)$ is
$[\gamma_1]+\cdots +[\gamma_{g+1}]$. Since the self-intersection
of the homology class $[\gamma_1]+\cdots +[\gamma_{g+1}]\in
H_1(M;\mathbb{Z}_2)$ is $g+1({\rm mod} 2)$,
$$\phi^*(\alpha\cup\alpha)=\phi^*(\alpha)\cup \phi^*(\alpha)=
(g+1)\beta,$$ where $0\neq\beta\in H^2(M;\mathbb{Z}_2)$. Hence the
mod $2$ degree of $\phi$ is $g+1({\rm mod} 2)$. Thus the mod 2
degree of $\tilde{\phi}$ is $g+1({\rm mod}2)$. This finishes the proof.
\qed\end{proof}
\begin{remark}
If we only consider a continuous map $\phi$ from the
non-oriented quotient surface $M$ to $\mathbb{R}P^2$,
then the conclusion is not necessarily true. The simplest
counter-example is a constant map. On the other hand, if
we assume that $\phi$ is a branched covering map, then the conclusion
is one part of Meeks' Theorem $1$ in \cite{Meeks}.
We don't know whether our conclusion could be generalized to the case of odd mapping $\tilde{\phi}:\widetilde{M}_1\to \widetilde{M}_2$ where each closed oriented surface is endowed with an
orientation-reversing involution without fixed points.
\end{remark}

\vspace{5mm} \noindent Xiang Ma, {\small\it LMAM, School of
Mathematical Sciences, Peking University, 100871 Beijing, People's
Republic of China. e-mail: {\sf maxiang@math.pku.edu.cn}}

\vspace{5mm} \noindent Peng Wang, {\small Department of Mathematics,
Tongji University, 200092.  Shanghai, People's Republic of China.
e-mail: {\sf netwangpeng@tongji.edu.cn}
\end{document}